\documentclass[11pt]{amsart}

\usepackage{amsmath,amsfonts,amssymb}
\usepackage{latexsym}
\usepackage[all]{xy}

\newtheorem{Thm}{Theorem}[section]
\newtheorem{Lem}[Thm]{Lemma}

\newtheorem{Rmk}[Thm]{Remark}
\newtheorem{Cor}[Thm]{Corollary}
\newtheorem{Def}[Thm]{Definition}

\newtheorem{Ques}[Thm]{Question}

\begin{document}

\title{$T$-structure and the Yamabe invariant}
\author{Chanyoung Sung}

\date{\today}

\address{Dept. of Mathematics and Institute for Mathematical Sciences \\
Konkuk University\\
         1 Hwayang-dong, Gwangjin-gu, Seoul, KOREA}
\email{cysung@kias.re.kr}
\thanks{This work was supported by the National Research Foundation of Korea(NRF) grant funded by the Korea government(MEST). (No. 2010-0016526, 2010-0001194)}

\maketitle \markboth{ Chanyoung Sung  } { $T$-structure and the Yamabe invariant  }

\begin{abstract}
The Yamabe invariant is an invariant of a closed smooth manifold,
which contains information about possible scalar curvature on it. It
is well-known that a product manifold $T^m\times B$ where $T^m$ is
the $m$-dimensional torus, and $B$ is a closed spin manifold of
nonzero $\hat{A}$-genus has zero Yamabe invariant. We generalize it
to various $T$-structured manifolds, for example $T^m$-bundles over
such $B$ whose transition functions take values in $Sp(m,\Bbb Z)$
(or $Sp(m-1,\Bbb Z)\oplus \{\pm 1\}$ for odd $m$).
\end{abstract}
\maketitle

\section{Introduction to Yamabe invariant}

The Yamabe invariant is an invariant of a smooth closed manifold
depending on its smooth topology.

Let $M$ be  a smooth closed manifold of dimension $n$. Given a smooth Riemannian metric $g$, a conformal class $[g]$ of $g$ is defined as
$$[g]=\{\varphi g \mid \varphi:M \rightarrow \Bbb
R^+ \  \textrm{is smooth} \}.$$

The famous Yamabe problem (\cite{lee}) states that there exists a
metric $\tilde{g}$ in $[g]$ which attains the minimum
$$\inf_{\tilde{g} \in
[g]} \frac{\int_M s_{\tilde{g}}\ dV_{\tilde{g}}}{(\int_M
dV_{\tilde{g}})^{\frac{n-2}{n}}},$$ where $s_{\tilde{g}}$ and
$dV_{\tilde{g}}$ respectively denote the scalar curvature and the
volume element of $\tilde{g}$.

It turns out that when $n\geq 3$, a unit-volume minimizer $\tilde{g}$ in $[g]$ has constant scalar curvature, which is equal to the above minimum value called the Yamabe constant of $[g]$ and denoted by $Y(M,[g])$.

It is known that the Yamabe constant is bounded above by $Y(S^n,[g_0])$ where $[g_0]$ denotes a standard round metric. Thus following a min-max procedure we define the Yamabe invariant $$Y(M):=\sup_{[g]} Y(M,[g])$$ of $M$.

The following facts are noteworthy.
\begin{itemize}
\item $Y(M)>0$ iff M admits a metric of positive scalar curvature.
\item If $M$ is simply connected and $\dim M\geq 5$, then $Y(M)\geq 0$.
\item For $r\in [\frac{n}{2},\infty]$, $$|Y(M,[g])|=\inf_{\tilde{g}\in [g]}(\int_{M} {|s_{\tilde{g}}|}^{r}
d\mu_{\tilde{g}})^{\frac{1}{r}}(\textrm{Vol}_{\tilde{g}})^{\frac{2}{n}-\frac{1}{r}},$$ where the infimum is attained only by the Yamabe minimizers.
\item When $Y(M,[g])\leq 0$,
$$Y(M,[g]) =
-\inf_{\tilde{g}\in [g]}(\int_{M} |s^{-}_{\tilde{g}}|^{r}
d\mu_{\tilde{g}})^{\frac{1}{r}}(\textrm{Vol}_{\tilde{g}})^{\frac{2}{n}-\frac{1}{r}},$$ where $s_{g}^-$ is defined as $\min\{s_g,0\}$.

Therefore  when $Y(M)\leq 0$,
\begin{eqnarray*}\label{form1}
Y(M)&=& -\inf_{g}(\int_{M} |s_{g}|^{r}
d\mu_{g})^{\frac{1}{r}}(\textrm{Vol}_{\tilde{g}})^{\frac{2}{n}-\frac{1}{r}}\\
&=& -\inf_{g}(\int_{M} |s_{g}^-|^{r}
d\mu_{g})^{\frac{1}{r}}(\textrm{Vol}_{\tilde{g}})^{\frac{2}{n}-\frac{1}{r}},
\end{eqnarray*}

so that $Y(M)$ measures how much negative scalar curvature is inevitable on $M$.
\item As an application of the above formula, if $M$ has an $F$-structure which will be explained in a later section,
$M$ admits a sequence of metrics with volume form converging to zero
while the sectional curvature are bounded below, so that $Y(M)\geq
0.$ (See \cite{pp}.)
\end{itemize}

\section{Computation of Yamabe invariant}

We now discuss how to compute the Yamabe invariant. When $M$ is a closed oriented surface, by the Gauss-Bonnet theorem $$Y(M)=4\pi \chi(M),$$ where $\chi(M)$ is the Euler characteristic of $M$.

When $M$ is a closed oriented 3-manifold, the Ricci flow gave many
answers in case $Y(M)\leq 0$ by the proof of geometrization theorem
due to G. Perelman.(See \cite{ander}). For example, $Y(\Bbb H^3/
\Gamma)$ is realized by the hyperbolic metric.

When $\dim M=4$, the Seiberg-Witten theory enables us to compute the
Yamabe invariant of K\"ahler surfaces through the Weitzenb\"ock
formula. LeBrun \cite{lb1,lb2,lb3} has shown that if $M$ is a
compact K\"ahler surface whose Kodaira dimension  is not equal to $-
\infty$,  then
$$Y(M)=-4\sqrt{2}\pi\sqrt{2\chi(\tilde{M})+3\tau(\tilde{M})},$$ where $\tau$ denotes the
signature and $\tilde{M}$ is the minimal model of $M$, and for $\Bbb
CP^{2}$, $$Y(\Bbb CP^{2})=4\sqrt{2}\pi\sqrt{2\chi(\Bbb
CP^{2})+3\tau(\Bbb CP^{2})}.$$

In higher dimensions, few examples have been computed so far, such
as $$Y(S^1\times S^{n-1})=Y(S^n)=
n(n-1)(\textrm{vol}(S^n(1)))^{\frac{2}{n}},$$ where $S^n(1)$ is the
unit sphere in $\Bbb R^{n+1}$, and $$Y(T^n)=Y(T^n\times
H)=Y(T^n\times B)=0,$$ where $H$ is a closed Hadarmard-Cartan
manifold, i.e. one with a metric of nonpositive sectional curvature,
and $B$ is a closed spin manifold with nonzero $\hat{A}$-genus.
These $T^n$-bundles have such property, because they admit a
T-structure and never admit a metric of positive scalar curvature by
Gromov-Lawson's enlargeability method \cite{GL,LM}.

We call a closed $n$-manifold $M$ {\it{enlargeable}} if the following holds: for any $\epsilon > 0$ and any Riemannian metric $g$ on $M$, there exists a Riemannian spin covering manifold $\tilde{M}$ of $(M,g)$ and an $\epsilon$-contracting map $f:\tilde{M}\rightarrow S^n(1)$, which is constant outside a compact subset of $\tilde{M}$ and of nonzero $\hat{A}$-degree defined as $\hat{A}(f^{-1}(\textrm{any regular value of }f))$.

Here a smooth map $F$ is called $\epsilon$-contracting if the norm of $DF$ is less than $\epsilon$. By using the Weitzenb\"ock formula for an appropriate twisted Dirac operator, they showed that such manifolds never admit a metric of positive scalar curvature.

They also generalized this to so-called weakly-enlargeable manifolds, where ``$\epsilon$-contracting'' is replaced by ``$\epsilon$-contracting on $2$-forms'' meaning that the induced map of $DF$ on tangent bi-vectors, i.e. a section of $\Lambda^2 (TM)$ has norm less than $\epsilon$.

$\hat{A}$-genus of a closed spin manifold $M$ is the integral over
$M$ of $$\hat{A}(TM):=1- \frac{p_1}{24}+\frac{-4p_2+7p_1^2}{5760}+
\cdots ,$$ where $p_i\in H^{4i}(M,\Bbb Z)$ is the $i$-th
Pontryagin class of $TM$. An important fact is that a closed spin
manifold with a metric of positive scalar curvature has zero
$\hat{A}$-genus.

Then a natural question for us to explore is
\begin{Ques}
Let $M$ be a $T^m$-bundle over a closed spin manifold $B$ with
nonzero $\hat{A}$-genus.  Is $Y(M)$ equal to zero ?
\end{Ques}

\section{$T$-structure}
An \emph{F-structure} which was introduced by Cheeger and Gromov
\cite{cg1,cg2} generalizes an effective $T^m$-action for $m\in \Bbb N$.
\begin{Def}
An F-structure on a smooth manifold is given by data
$(U_i,\hat{U_i},T^{k_i})$ with the following conditions:
\begin{enumerate}
    \item $\{U_i\}$ is a locally finite open cover.
    \item $\pi_i:\hat{U_i}\mapsto U_i$ is a finite Galois covering with
covering group $\Gamma_i$.
    \item A torus $T^{k_i}$ of dimension $k_i$ acts effectively on
$\hat{U_i}$ in a $\Gamma_i$-equivariant way, i.e. $\Gamma_i$ also
acts on $T^{k_i}$ as an automorphism so that
$$\gamma(gx)=\gamma(g)\gamma(x)$$ for any $\gamma\in
\Gamma_i,$ $g\in T^{k_i},$ and $x\in \hat{U_i}$.
    \item If $U_i\cap U_j \neq \emptyset$, then there is a common
covering of $\pi_i^{-1}(U_i\cap U_j)$ and $\pi_j^{-1}(U_i\cap U_j)$ such that it is invariant under the lifted actions of $T^{k_i}$ and $T^{k_j}$, and they commute.
\end{enumerate}
\end{Def}
As a special case, a $T$-structure is an $F$-structure in which all
the coverings $\pi_i$'s are trivial.

Typical examples of $T$-structure are torus bundles.
\begin{Thm}\label{prev}
Any $T^m$-bundle over a smooth manifold whose transition functions are
$T^m\rtimes GL(m,\Bbb Z)$-valued has a $T$-structure. In particular,
any $S^1$ or $T^2$-bundle has a $T$-structure.
\end{Thm}
\begin{proof}
Here $T^m$ acts by translation, and hence
the transition functions are affine maps at each fiber
direction. Obviously the local $T^m$ actions along the fiber are commutative
on the intersections to give a global $T$-structure.

The second statement follows from the well-known fact that the
diffeomorphism group of  $T^m$ for $m=1,2$ is homotopically
equivalent to $T^m\rtimes GL(m,\Bbb Z)$. where Thus we may assume
that the transition functions are $T^m\rtimes GL(m,\Bbb
Z)$-valued.
\end{proof}

Other typical examples are manifolds with a nontrivial smooth $S^1$
action. Such examples we will use are projective spaces such as
$\Bbb RP^n, \Bbb CP^n, \Bbb HP^n$, and $CaP^2$. (For the case of the
Cayley plane which actually has an $S^3$-action, see \cite{AB}.)

Finally, an important fact is :
\begin{Thm}[Paternain and Petean \cite{pp}]
Suppose $X$ and $Y$ are $n$-manifolds with $n > 2$, which admit a
$T$-structure. Then $X\#Y$ also admits a $T$-structure.
\end{Thm}

\section{Main Results}

Motivated by Gromov-Lawson's enlargeability technique, we prove :
\begin{Thm}\label{th0}
Let $B$ be a closed spin manifold of dimension $4d$ with nonzero
$\hat{A}$-genus, and $M$ be a $T^m$-bundle over $B$ whose transition
functions take values in $Sp(m,\Bbb Z)$ (or $Sp(m-1,\Bbb Z)\oplus
\{\pm 1\}$ for odd $m$).  Then
$$Y(M)=0.$$
\end{Thm}
\begin{proof}
By theorem \ref{prev}, $M$ has a $T$-structure so that $Y(M)\geq 0$.

We only have to show that $M$ never admits a metric of positive
scalar curvature. To the contrary, suppose that it admits such a
metric $h$, and we will derive a contradiction. The basic idea is to
apply the Bochner-type method to a twisted Spin$^c$ bundle on $M$
whose topological index is nonzero.

First, we consider the case when  $m$ is even, say $2k$. Let
$\Lambda$ denote a lattice in $\Bbb R^{2k}$ so that $T^{2k}= \Bbb
R^{2k}/ \Lambda$. Take an integer $n \gg 1$. There is an obvious
covering map from $\Bbb R^{2k}/ n\Lambda$ onto $\Bbb R^{2k}/
\Lambda$ of degree $n^{2k}$, and we claim that this covering map
can be extended to all the fibers in $M$ to give a covering $p:
M_n \rightarrow M$. The following lemma justifies this :
\begin{Lem}
The same transition functions as $\Bbb R^{2k}/ \Lambda$-bundle $M$
give $\Bbb R^{2k}/ n\Lambda$-bundle $M_n$ with the covering
projection $p$.
\end{Lem}
\begin{proof}
For a transition map $g_{\alpha\beta}\in Sp(2k,\Bbb Z)$ downstairs,
the same transition map $g_{\alpha\beta}$ upstairs is the unique
lifting map which satisfies $p\circ
g_{\alpha\beta}=g_{\alpha\beta}\circ p$ and sends 0 to 0.

It only needs to be proved that the transition maps satisfy the
axioms for the bundle, in particular the axiom
$g_{\beta\gamma}\circ g_{\alpha\beta}=g_{\alpha\gamma}$. This
follows from the uniqueness of the lifting  map sending 0 to 0.
(In fact, this cocycle condition holds without modulo $\Bbb Z$.)
\end{proof}

We endow $M_n$ with a metric $ h_n:=p^* h$.

\begin{Lem}
 There exists a closed 2-form $\omega$ on $M_n$ such that $\omega^{k+1}=0$
 and it restricts to a generator of $H^2(T^{2k},\Bbb Z)$
 at each fiber $T^{2k}$.
\end{Lem}
\begin{proof}
For each $U\times T^{2k}$ where $U$ is an open ball in $B$, take
$\omega$ to be a standard symplectic form of $T^{2k}$ representing
a generator of $H^2(T^{2k},\Bbb Z)$. Since $\omega$ is invariant
under $Sp(2k,\Bbb Z)$, it is globally defined on $M_n$. (Note that
the transition functions are locally constant.) Obviously
$\omega^{k+1}=0$ at each point.
\end{proof}

Let $E$ be the complex line bundle on $M_n$ whose first Chern
class is $[\omega]$. Take a connection $A^E$ of $E$ whose
curvature 2-form $R^E=dA^E$ is equal to $-2\pi i \omega$.

We claim that $$|R^E|_{h_n}\rightarrow 0\ \ \ \ \ \ \ \textrm{as}\ \
n\rightarrow \infty.$$
\begin{Lem}
  $|R^E|_{h_n}= O(\frac{1}{n^2})$.
\end{Lem}
\begin{proof}
For a local coordinate $(x_1,\cdots,x_{4d},y_1,\cdots,y_{2k})$ of
$B\times T^{2k}$, $$\omega = dy_1\wedge dy_2+\cdots +
dy_{2k-1}\wedge dy_{2k}.$$ We will show that
$|dy_{\mu}|_{h_n}=O(\frac{1}{n})$ for all $\mu$. First,
$$h_n(\frac{\partial}{\partial x_i},\frac{\partial}{\partial x_j})=h(\frac{\partial}{\partial x_i},\frac{\partial}{\partial x_j}),$$
$$h_n(\frac{\partial}{\partial x_i},\frac{\partial}{\partial y_{\mu}})=nh(\frac{\partial}{\partial x_i},\frac{\partial}{\partial y_{\mu}}),$$
 $$h_n(\frac{\partial}{\partial y_{\mu}},\frac{\partial}{\partial y_{\nu}})=n^2h(\frac{\partial}{\partial y_{\mu}},\frac{\partial}{\partial y_{\nu}})$$ for all $i,j,\mu$, and $\nu$. Thus
$$h_n= \left(
   \begin{array}{cc}
     O(1) & O(n) \\
     O(n) & O(n^2) \\
   \end{array}
 \right),
$$ where the block division is according to the division by $x$ and $y$ coordinates, and
\begin{eqnarray*}
(h_n)^{-1}&=&\frac{1}{\det(h_n)}\textrm{adj}(h_n)\\ &=& \frac{1}{O(n^{4k})}\left(
    \begin{array}{cc}
      O(n^{4k}) & O(n^{4k-1}) \\
      O(n^{4k-1}) &  O(n^{4k-2}) \\
    \end{array}
  \right)\\ &=& \left(
    \begin{array}{cc}
      O(1) & O(\frac{1}{n}) \\
      O(\frac{1}{n}) &  O(\frac{1}{n^2}) \\
    \end{array}
  \right),
\end{eqnarray*}
which means $|dx_i|_{h_n}=O(1)$ and $|dy_{\mu}|_{h_n}=
O(\frac{1}{n})$ for all $i$ and $\mu$, completing the proof.
\end{proof}

In order to use the Bochner argument, we need to show that $M_n$ is
spin$^c$. Using the orthogonal decomposition by $h_n$,
$$TM_n=V\oplus H=V\oplus \pi^*(TB),$$ where $V$ and $H$ respectively denote the vertical and horizontal space,
and $\pi: M_n\rightarrow B$ be the torus bundle projection.
Obviously $H$ is spin, because $B$ is spin. Since $V$ is a symplectic
$\Bbb R^{2k}$-vector bundle, it admits a compatible almost-complex structure, and hence it can be viewed as a $\Bbb C^k$-vector
bundle. Thus
$$w_2(M_n)=w_2(V)+w_2(H)\equiv c_1(V) \ \ \textrm{mod}\ 2$$ meaning
that $M_n$ is spin$^c$. Let $S$ be the associated vector bundle to
the Spin$^c$ bundle over $M_n$ obtained using $\sqrt{\det_{\Bbb C} V}$.

Consider a twisted spin$^c$ Dirac operator $D^E$ on $S\otimes E$
where $E$ is equipped with a connection $A^E$. The Weitzenb\"ock
formula says that
$$(D^E)^2=\nabla^*\nabla +\frac{1}{4}s_{h_n}+\frak{R}^E.$$
Here $\frak{R}^E(\sigma\otimes v) = \sum_{i<j}(e_ie_j\sigma)\otimes (R^E_{e_i,e_j}v)$
where $\{e_i\}$ is an orthonormal frame for $(M_n,h_n)$. Note that
$$|\frak{R}^E|_{h_n}\leq C |R^E|_{h_n}$$ where $C$ is a positive constant depending on the dimension of $M$.
By taking $n$ sufficiently large, we can ensure that $s_{h_n}>
|\frak{R}^E|_{h_n}$ everywhere, and hence $\ker D^E = 0$. Thus the
index of the operator
$$D^E_+ : \Gamma(S_+\otimes E)\rightarrow \Gamma(S_-\otimes E)$$ is
$$\ker D^E|_{S_+\otimes E}- \ker D^E|_{S_-\otimes E}=0,$$
where $S_{\pm}$ respectively denotes the plus and negative spinor
bundle.

On the other hand, we can also compute the index using the
Atiyah-Singer index theorem  \cite{LM}. Note that $V$ has locally
constant transition functions, and hence can be given a flat
connection. This implies that $\hat{A}(V)=1$ so that
\begin{eqnarray*}
\textrm{index}(D^E_+)&=& \{ch (E) \cdot \hat{A}(TM_n)\}[M_n]\\
&=& \{(1+[\omega]+\cdots +\frac{1}{k!}[\omega]^k )\cdot\hat{A}(V)\cdot\hat{A}(\pi^*(TB))\}[M_n]\\
&=& \{(1+[\omega]+\cdots +\frac{1}{k!}[\omega]^k)\cdot\pi^*(\hat{A}(B))\}[M_n]\\
&=& \{\frac{1}{k!}[\omega]^k\cdot\pi^*(\hat{A}_d(B))\}[M_n]\\
&=& \int_{\pi^{-1}(PD[\hat{A}_d(B)])}\frac{1}{k!}[\omega]^k\\
&=& (\hat{A}(B)[B])\int_{\pi^{-1}(pt)}\frac{1}{k!}[\omega]^k\\
&\ne& 0,
\end{eqnarray*}
which yields a contradiction.

The odd $m$ case is reduced to the even case. If $m$ is odd,
consider an an $S^1$-bundle over $M$ with transition functions
exactly equal to the transition functions $\{\pm 1\}$ of the last
$S^1$-factor of $T^{m}$ in $M$ over $B$. Then $M'$ is a
$T^{m+1}$-bundle over $B$ with transition functions taking values
in $Sp(m+1,\Bbb Z)$. We put a locally product metric on $M'$. Then
it also has positive scalar curvature, yielding a contradiction to
the above even case.

\end{proof}

\begin{Thm}\label{th1}
Let $B$ be a closed spin manifold of dimension $4d$ with nonzero
$\hat{A}$-genus, and $M$ be an $S^1$ or $T^2$-bundle over $B$ whose
transition functions take values in $GL(1,\Bbb Z)$ or $GL(2,\Bbb Z)$
respectively. Then
$$Y(M)=0.$$
\end{Thm}
\begin{proof}
Again by theorem \ref{prev}, $M$ has a $T$-structure so that $Y(M)\geq 0$.
It remains to show $M$ never admits a metric of positive scalar curvature,
and let's assume it does.

First, the case of $S^1$ bundle can be reduced to the case of
$T^2$ bundle by considering a Riemannian product $M\times S^1$
which also has positive scalar curvature. From now on, we consider
the case of $T^2$ bundle.

Secondly, we may also assume that $M$ is orientable, i.e. the
transition functions for the torus bundle are
orientation-preserving. Otherwise, we consider $\bar{M}$ from the
lemma below, which also admits a metric of positive scalar
curvature by lifting the metric of $M$.
\begin{Lem}
There exists a finite covering $\bar{M}$ of $M$ such that $\bar{M}$
is an orientable $T^2$-bundle over a closed spin manifold of nonzero
$\hat{A}$-genus.
\end{Lem}
\begin{proof}
Let $\hat{B}$ be the universal cover of $B$, and $\hat{M}$ be the
manifold obtained by lifting the torus bundle over $B$ to
$\hat{B}$. Since $\hat{B}$ is simply-connected, $\hat{M}$ is
orientable, and $\pi_1(B)$ acts on $\hat{M}$ to give $M=\hat{M}/
\pi_1(B)$.

Let $G$ be a subset of $\pi_1(B)$, which consists of elements
preserving orientation of the fiber torus. Then $G$ is a subgroup
of index 2. Thus $\hat{M}/G$ is an orientable  $T^2$-bundle over
$\hat{B}/G$ which is a double cover of $B$ so that it is  also
spin with nonzero $\hat{A}$-genus.
\end{proof}

Now if $M$ is orientable, its transition functions take values in $SL(2,\Bbb Z)=Sp(2,\Bbb Z)$ so that
the previous theorem can be applied to derive a contradiction.
\end{proof}

\begin{Rmk}
In fact,  theorem \ref{th0} holds for any $T^m$-bundle with
$T^m\rtimes GL(m,\Bbb Z)$-valued transition functions, which has a
finite covering diffeomorphic to $M$ as in theorem \ref{th0}.
\end{Rmk}

Combining our results with the previous results in
\cite{sung1}, we can compute more general $T$-structured manifolds :
\begin{Cor}
Let $M$ be a $T^m$-bundle in all the above so that $Y(M)=0$. If
$\dim M=4n$, then
$$Y(M\sharp\ k\ \Bbb HP^n\sharp \ l\ \overline{\Bbb HP^n})=0,$$ and  if
$\dim M=16$, then
$$Y(M\sharp\ k\ \Bbb HP^4\sharp \ l\ \overline{\Bbb HP^4}\sharp\ k'\ CaP^2\sharp \ l'\ \overline{CaP^2})=0,$$ where $k,l,k'$, and $l'$ are nonnegative integers, and the overline denotes the reversed orientation.
\end{Cor}

Finally we remark that theorem \ref{th1} can be extended to the case
when $\dim B\leq 2$ by applying the Seiberg-Witten theory :
\begin{Thm}
Let $B$ be a circle or a closed oriented surface, and $X$ be an $S^1$
or $T^2$-bundle over $B$. Suppose that $X$ or $X\times T^m$ for $m=1,2$ has a finite cover $M$ with $b_2^+(M) >1$ which is a $T^2$-bundle over a surface whose
transition functions take values in a discrete subgroup of
$T^2\rtimes SL(2,\Bbb Z)$. Then
$$Y(X)=0.$$
\end{Thm}
\begin{proof}
It suffices to show that $M$ never admits a metric of positive scalar
curvature.

Using the 2-form $\omega$ on $M$
which restricts to a standard symplectic form on each fiber, we have a
symplectic form $\pi^*\sigma+\omega$ on $M$,
where $\sigma$ is a symplectic form of $B$, and $\pi:
M\rightarrow B$ is the bundle projection.

Then the Seiberg-Witten invariant of  the canonical Spin$^c$ structure of
$M$ is 1 so that it never admits a metric of positive scalar curvature.
\end{proof}

\bigskip



\begin{thebibliography}{99}

\bibitem{ander} M. T. Anderson, {\em Canonical metrics on 3-manifolds and 4-manifolds},
Asian J. Math. {\bf 10} (2006), 127--163.

\bibitem{AB} M. Atiyah and J. Berndt,
{\em Projective planes, Severi varieties and spheres}, Surv. Differ.
Geom. {\bf Vol. VIII} (2003), 1--27.


\bibitem{cg1} J. Cheeger and M. Gromov,
{\em  Collapsing Riemannian manifolds while keeping their curvature
bounded I}, J. Diff. Geom. {\bf 23} (1986), 309--346.

\bibitem{cg2} J. Cheeger and M. Gromov,
{\em  Collapsing Riemannian manifolds while keeping their curvature
bounded II}, J. Diff. Geom. {\bf 32} (1990), 269--298.

\bibitem{GL} M. Gromov and H. B. Lawson, {\em Positive scalar curvature and the Dirac operator on complete riemannian manifolds}, Publ. Math. IHES, {\bf 58} (1983), 83--196.


\bibitem{Gur} M.J. Gursky and C. LeBrun,
{\em Yamabe constants and spin$^c$ structures}, GAFA {\bf 8} (1998), 965--977.


\bibitem{IL} M. Ishida and C. LeBrun,
{\em Curvature, connected sums, and Seiberg-Witten theory}, Comm.
Anal. Geom. {\bf 11} (2003), 809--836.

\bibitem{koba} O. Kobayashi, {\em Scalar curvature of a metric with unit
volume}, Math. Ann. {\bf 279} (1987), 253--265.

\bibitem{LM} H. B. Lawson and M. L. Michelson,
Spin Geometry, Priceton University Press (1989).

\bibitem{lb1} C. LeBrun,
{\em Four manifolds without Einstein metrics}, Math. Res. Lett.
{\bf 3} (1996), 133--147.

\bibitem{lb2} C LeBrun, {\em Yamabe constants and the perturbed Seiberg-Witten
equations}, Comm. Anal. Geom.  {\bf 5} (1997), 535--553.

\bibitem{lb3} C LeBrun, {\em Kodaira dimension and the Yamabe problem},
Comm. Anal. Geom.  {\bf 7} (1999), 133--156.


\bibitem{lee} J. Lee and T. Parker, {\em The Yamabe Problem},
Bull. Amer. Soc. {\bf 17} (1987), 37--81.

\bibitem{pp} G. P. Paternain and J. Petean, {\em Minimal entropy and
collapsing with curvature bounded from below}, Invent. Math. {\bf
151} (2003), 415--450.

\bibitem{PY} J. Petean and G. Yun, {\em Surgery and the Yamabe invariant},
GAFA {\bf 9} (1999), 1189--1199.


\bibitem{sung1} C. Sung, {\em Connected sums with $\Bbb HP^{n}$ or $CaP^{2}$  and the Yamabe invariant}, arXiv:0710.2379.

\bibitem{sung2} C. Sung, {\em Surgery and equivariant Yamabe invariant},
Diff. Geom. Appl. {\bf 24} (2006), 271--287.

\bibitem{sung3} C. Sung, {\em Surgery, Yamabe invariant,and Seiberg-Witten theory},
J. Geom. Phys. {\bf 59} (2009), 246--255.


\end{thebibliography}
\end{document}